\documentclass[final,3p,times]{elsarticle}

\usepackage{amsmath,amssymb,amsfonts,mathptmx,dsfont}
\usepackage{graphicx,graphics,subfigure,color}
\newtheorem{theorem}{Theorem}[section]
\newtheorem{lemma}{Lemma}[section]

\newdefinition{definition}{Definition}[section]
\newproof{proof}{Proof}
\newdefinition{assumption}{Assumption}[section]
\numberwithin{equation}{section}
\def\d {{\rm d}}

\begin{document}

\begin{frontmatter}



\title{Traveling waves for SIR model on two-dimensional lattice}

\author[HLJU1,HLJU2]{Ran Zhang}
\ead{ranzhang@hlju.com}

\author[HLJU1,HLJU2]{Shunchang Su}
\ead{2241491@s.hlju.edu.cn}

\author[HLJU1,HLJU2]{Xue Ren\corref{cor1}}
\ead{xueren@hlju.edu.cn}
\cortext[cor1]{Corresponding author.}

\affiliation[HLJU1]{organization={School of Mathematical Sciences, Heilongjiang University}, city={Harbin}, postcode={150080}, country={China}}

\affiliation[HLJU2]{organization={Heilongjiang Provincial Key Laboratory of the Theory and Computation of Complex Systems}, city={Harbin}, postcode={150080}, country={China}}



\begin{abstract}
In this study, we investigate the existence of traveling wave solutions for a SIR model on two-dimensional lattice. The existence of traveling waves is established within the framework of upper and lower solutions and the Schauder fixed-point theorem. Moreover, we construct a Lyapunov functional to analyze the asymptotic behavior of the traveling wave solutions. This is a challenging task due to the two-dimensional lattice structure.
\end{abstract}


\end{frontmatter}



\section{Introduction}
\def\d {{\rm d}}
Traditionally, ordinary differential equation (ODE) models provide a clear and effective framework for analyzing the dynamics of infectious diseases. ODE models determine whether a disease dies out or persists by analyzing equilibrium states, but they do not describe how infection spreads in space. In contrast, reaction-diffusion models with traveling wave solutions characterize the spatial invasion process, including whether a localized outbreak can propagate into new regions and at what speed \cite{DiekmannJMB1978}. A traveling waves also representing the advancing front of geographic spread by connecting the disease-free and endemic steady states. Thus, the study of traveling wave solutions in reaction-diffusion equations for infectious disease models offers significant advantages in understanding the spread and control of diseases. 

In diffusive infectious disease modelling, local diffusion and nonlocal diffusion describe how a disease spreads across space continually, but they differ in the range of interactions they considered. Local diffusion assumes that the spread is continuous and restricted to neighboring regions, with the disease propagating from one area to the next based on the proximity of infected individuals. In contrast, nonlocal diffusion extends this idea by allowing interactions and disease spread to occur over larger distances. It is particularly useful for modeling long-range interactions, such as vector-borne diseases, where the disease can spread across vast distances due to environmental factors or movement of vectors. For infectious disease models with local and nonlocal diffusion, traveling wave solutions have been extensively studied (see \cite{ZhaoWangRuanJMB2018,WangZhaoWangZhangJDDE2023,ZhaoLiuNARWA2026} for local diffusive epidemic models, \cite{LiLiYangAMC2014,LiLinMaYangDCDSB2014,HeZhangMMAS2025} for nonlocal diffusive epidemic models and the references therein). 

Building on these continuous models, discrete diffusion can be introduced through a lattice system, where individuals or pathogens are confined to a grid, and the disease spreads in discrete steps between distinct spatial points. This approach is advantageous in representing spatially structured environments, such as cities or networks, where interactions occur in specific, discrete locations. The primary benefit of discrete diffusion is its ability to model more realistic, structured scenarios where disease spread follows distinct pathways, such as through transportation systems or across regions with different infection levels, offering a finer-grained representation of disease dynamics. Studying traveling wave solutions in discrete diffusion infectious disease models helps reveal how diseases propagate across spatially structured or patchy environments. It also allows one to characterize the invasion speed and threshold conditions for spatial spread in networks or lattice-like habitats.

In 2016, Fu et al. \cite{FuGuoWuJNCA2016} proposed an diffusive SIR model on one-dimensional (1-D) lattice as follows:
\begin{equation}
\label{PreModel2}
\left\{
\begin{array}{l}
\vspace{2mm}
\displaystyle  \frac{\textrm{d} S_n(t)}{\textrm{d} t} = [S_{n+1} + S_{n-1} - 2S_n] - \beta S_nI_n, \\
\displaystyle  \frac{\textrm{d} I_n(t)}{\textrm{d} t} = d[I_{n+1} + I_{n-1} - 2I_n] + \beta S_nI_n - \gamma I_n, \\
\end{array}\right.
\end{equation}
where \(n \in \mathds{Z}\). The state variables \(S_n\) and \(I_n\) represent the population densities of susceptible and infected individuals at the grid point \(n\) at time \(t\), and \(d\) denotes the diffusion rate of the infected population. In \cite{FuGuoWuJNCA2016}, the authors used nonlinear analysis techniques to establish the existence of traveling wave solutions for model (\ref{PreModel2}) when the basic reproduction number exceeds 1. Subsequently, Wu \cite{WuJDE2017} investigated the existence of critical traveling wave solutions for model (\ref{PreModel2}). By incorporating population dynamics such as birth and death rates, Chen et al. \cite{ChenGuoHamelNon2017} extended model (\ref{PreModel2}) and obtained the existence of semi-traveling wave solutions. Furthermore, Zhang et al. \cite{ZhangWangLiuJNS2021} resolved the strong traveling wave solutions for the model proposed in \cite{ChenGuoHamelNon2017} by using Lyapunov functional techniques.

The discrete diffusive infectious disease models discussed above are all based on a 1-D lattice. While these models provide valuable insights into the spread of diseases in structured populations, they are inherently limited in capturing the complexities of real-world environments. In particular, a 1-D lattice only allows interactions along a single spatial axis, making it difficult to model more complex interactions that naturally occur in two-dimensional (2-D) spaces, such as the spread of disease in urban areas or across larger geographic regions. In \cite{VleckParetCahnSIAP1998,GuoWuOJM2008}, a infinite system with different nonlinear terms on the two-dimensional integer lattice were studied, such model takes the following form
\begin{equation}\label{PreModel1}
\frac{\textrm{d} u_{i,j}(t)}{\textrm{d} t} = \alpha \left[u_{i+1,j}(t)+u_{i-1,j}(t)+u_{i,j+1}(t)+u_{i,j-1}(t)-4u_{i,j}(t)\right] - f(u_{i,j}(t)),\ \ i,j\in \mathds{Z}.
\end{equation}
In \cite{VleckParetCahnSIAP1998}, the authors explored model \eqref{PreModel1} with bistable nonlinearity and addresses propagation failure in the 2-D lattice. Meanwhile, the authors in \cite{GuoWuOJM2008} discussed model \eqref{PreModel1} with monostable nonlinearity, and demonstrating the existence and uniqueness of traveling wave solutions, showing that waves exist above a minimal speed and that the wave profiles are strictly monotone. Thereafter, an increasing number of scholars are focusing on the traveling wave solutions of such equations with different factors, see \cite{ChengLiWangIMAJAM2008,ChengLiWangDCDSB2010,YuZhangWangMCM2013,XuJDE2020} and the references therein. 


Biologically, using 2-D lattice in biological models offers clear advantages by realistically representing spatial interactions, where individuals or pathogens can move and spread in multiple directions, reflecting real-life scenarios with greater freedom of movement across both horizontal and vertical axes. Motivated by these advantages, we focus on studying the SIR model on 2-D lattice:
\begin{equation}
\label{Model}\left\{
\begin{array}{l}
\vspace{2mm}
\displaystyle   \frac{\textrm{d} S_{i,j}(t)}{\textrm{d} t} = d_1\left[S_{i+1,j}(t)+S_{i-1,j}(t)+S_{i,j+1}(t)+S_{i,j-1}(t)-4S_{i,j}(t)\right] + \Lambda - \frac{\beta S_{i,j}(t)I_{i,j}(t)}{1+\alpha I_{i,j}(t)} - \mu_1 S_{i,j}(t),\\
\vspace{2mm}
\displaystyle   \frac{\textrm{d} I_{i,j}(t)}{\textrm{d} t} = d_2\left[I_{i+1,j}(t)+I_{i-1,j}(t)+I_{i,j+1}(t)+I_{i,j-1}(t)-4I_{i,j}(t)\right] + \frac{\beta S_{i,j}(t)I_{i,j}(t)}{1+\alpha I_{i,j}(t)} - \mu_2 I_{i,j}(t),\\
\displaystyle   \frac{\textrm{d} R_{i,j}(t)}{\textrm{d} t} = d_3\left[R_{i+1,j}(t)+R_{i-1,j}(t)+R_{i,j+1}(t)+R_{i,j-1}(t)-4R_{i,j}(t)\right] + \gamma I_{i,j}(t)-\mu_1R_{i,j}(t)  .
\end{array}\right.
\end{equation}
where $S_{i,j}(t)$, $I_{i,j}(t)$ and $R_{i,j}(t)$ represents spatial densities of susceptible, infective and removed individuals at each spatial grid point \((i, j)\) and time $t \geq 0$, The biological significance of the parameters in the model is as follows:

\begin{itemize}
    \item \(d_1, d_2, d_3\): Diffusion coefficients for susceptible, infected, and recovered individuals, respectively. These represent the rate at which individuals move or disperse across neighboring spatial locations (e.g., due to migration or contact).
    \item \(\Lambda\): The recruitment rate of susceptible individuals, representing the inflow of new individuals into the susceptible class (e.g., births or immigration).
    \item \(\beta\): The transmission rate coefficient, which determines the probability of disease transmission upon contact between a susceptible and an infected individual.
    \item \(\alpha\): A nonlinearity factor in the transmission term, accounting for the saturation effect on disease transmission at high infection levels. This reflects the fact that as the number of infected individuals increases, the per-capita transmission rate decreases.
    \item \(\mu_1, \mu_2\): Mortality rates for susceptible and infected individuals, respectively. These represent the natural death rates of individuals in these compartments.
    \item \(\gamma\): The recovery rate, which indicates the rate at which infected individuals recover and transition into the recovered class.
\end{itemize}

Our model \eqref{Model} captures the spread of infectious diseases in a 2-D spatial domain, where individuals move and interact in both directions. By extending the traditional 1-D epidemic model \cite{ChenGuoHamelNon2017,ZhangWangLiuJNS2021} to 2-D case, we aim to explore the effect of disease spread across a broader range of interactions. This 2-D lattice SIR model \eqref{Model} offers a richer framework to study the dynamics of infectious diseases in real-world scenarios, providing new insights into the effectiveness of intervention strategies in more complex, spatially structured environments.

\section{Preliminaries }

In this section, owing to the lack of comparison principle property for the wave profile of system (\ref{Model}), we shall establish the existence of traveling wave solutions through the construction of ordered upper-lower solutions combined with Schauder's fixed point theorem. We investigate traveling wave solutions defined as bounded continuous profiles propagating along the spatial-temporal domain with constant wave speed~$c > 0$. Specifically, these solutions admit a functional representation in terms of the composite variable $i \cos\theta + j \sin \theta + ct$, satisfying the following conditions
\begin{equation}\label{dy2.1}
S_{i,j}(t)=S(i \cos\theta + j \sin \theta + ct)~~\mathrm{and}~~I_{i,j}(t)=I(i \cos\theta + j \sin \theta + ct).
\end{equation}
Next, assuming $\xi = i \cos\theta + j \sin \theta + ct$, we can rewrite system (\ref{Model}) as the following traveling wave system
\begin{equation}
\label{WaveEqu}\left\{
\begin{array}{l}
\vspace{2mm}
\displaystyle   c S'(\xi) =
d_1\mathfrak{J}[S](\xi) + \Lambda - \frac{\beta S(\xi)I(\xi)}{1+\alpha I(\xi)} - \mu_1 S(\xi),\\
\displaystyle   c I'(\xi) =
d_2\mathfrak{J}[I](\xi) + \frac{\beta S(\xi)I(\xi)}{1+\alpha I(\xi)} - \mu_2 I(\xi),
\end{array}\right.
\end{equation}
for all \(\xi\in\mathbb{R}\), where
\[
\mathfrak{J}[\phi](\xi) : = \phi(\xi + \sin\theta) + \phi(\xi - \sin\theta) + \phi(\xi + \cos\theta) + \phi(\xi - \cos\theta) - 4\phi(\xi).
\]
Then, the traveling wave solution we want to find has the following asymptotic boundary conditions
\begin{equation}\label{bj1}
\underset{\xi\rightarrow -\infty}{\lim}(S(\xi), I(\xi))=(S_{0},0),
\end{equation}
and
\begin{equation}\label{bj2}
\underset{\xi\rightarrow +\infty}{\lim}(S(\xi), I(\xi))=(S^{\ast},I^{\ast}).
\end{equation}
Among them, $(S_{0},0)$ is the disease-free equilibrium and $(S^{\ast},I^{\ast})$ is the endemic equilibrium, it follows from \cite{KorobeinikovBMB2006} that there exists only one positive $(S^{\ast},I^{\ast})$ if $\Re_0: = \frac{\beta S_0}{\mu_2} > 1$, and $\Re_0$ is the
basic reproduction number. Linearizing the second equation of the system (\ref{WaveEqu}) at the disease-free equilibrium, we can obtain the following equation
\begin{equation}\label{LinEqu}
c I'(\xi) =
d_2\mathfrak{J}[I](\xi) + \beta S_0I(\xi) - \mu_2 I(\xi).
\end{equation}
Let $I(\xi) = \mathrm{e}^{\lambda \xi}$, one has
\[
d_2 \left[\mathrm{e}^{\lambda \sin\theta} + \mathrm{e}^{-\lambda \sin\theta} + \mathrm{e}^{\lambda \cos\theta} + \mathrm{e}^{-\lambda \cos\theta}-4\right] - c\lambda + \beta S_0 - \mu_2 = 0.
\]
Denote
\begin{equation}\label{CharEqu}
\Delta_c(\lambda) : =
d_2 \left[\mathrm{e}^{\lambda \sin\theta} + \mathrm{e}^{-\lambda \sin\theta} + \mathrm{e}^{\lambda \cos\theta} + \mathrm{e}^{-\lambda \cos\theta}-4\right] - c\lambda + \beta S_0 - \mu_2.
\end{equation}
Fixing $c>0$ and $\lambda>0$, by some calculations, we get
\begin{align*}
&\Delta_c(0) = \beta S_0 - \mu_2 > 0\ \ \ \textrm{if}\ \ \ \Re_0: = \frac{\beta S_0}{\mu_2} > 1,\\
&\frac{\partial \Delta_c(\lambda)}{\partial\lambda} = d_2 \left[\sin\theta\mathrm{e}^{\lambda \sin\theta} - \sin\theta \mathrm{e}^{-\lambda \sin\theta} + \cos\theta\mathrm{e}^{\lambda \cos\theta} - \cos\theta \mathrm{e}^{-\lambda \cos\theta}-4\right] - c,\\
&\frac{\partial^2 \Delta_c(\lambda)}{\partial\lambda^2} = d_2 \left[\sin^2\theta\left(\mathrm{e}^{\lambda \sin\theta} + \mathrm{e}^{-\lambda \sin\theta}\right) + \cos^2\theta\left(\mathrm{e}^{\lambda \cos\theta} + \mathrm{e}^{-\lambda \cos\theta}\right)\right] > 0,\\
&\lim_{c\rightarrow+\infty} \Delta_c (\lambda) = -\infty,\ \ \ \frac{\partial \Delta_c(\lambda)}{\partial c}= -\lambda < 0,\ \ \ \frac{\partial \Delta_c(\lambda)}{\partial\lambda}\bigg|_{(0,c)} = -c < 0,
\end{align*}
Thus, we can obtain the following lemma.
\begin{lemma}\label{WaveSpeed}
Let $\Re_{0}>1.$ There exist $c^*>0$ and $\lambda^*>0$ such that
\[
\frac{\partial \Delta_c(\lambda)}{\partial \lambda}\bigg|_{(\lambda^*,c^*)} = 0\ \ \textrm{and}\ \ \Delta^{c^*}(\lambda^*) = 0.
\]
Furthermore,
\begin{description}
  \item[(i)] if $c=c^*,$ then $\Delta_c(\lambda)=0$ has only one positive real root $\lambda^*;$
  \item[(ii)] if $0<c<c^*,$ then $\Delta_c(\lambda)>0$ for all $\lambda\in(0, +\infty)$;
  \item[(iii)] if $c>c^*,$ then $\Delta_c(\lambda)=0$ has two positive real roots $\lambda_1,$ $\lambda_2$ with $\lambda_1<\lambda^*<\lambda_2$.
\end{description}
\end{lemma}
From Lemma \ref{WaveSpeed}, we have
\begin{equation}
\label{WaveSpeedRel}
\Delta_c(\lambda)\left\{
\begin{array}{l}
\vspace{2mm}
\displaystyle   >0\ \ \ {\rm for}\ \ \ \lambda<\lambda_1,\\
\vspace{2mm}
\displaystyle   <0\ \ \ {\rm for}\ \ \ \lambda_1<\lambda<\lambda_2,\\
\displaystyle   >0\ \ \ {\rm for}\ \ \ \lambda>\lambda_2.
\end{array}\right.
\end{equation}
We always fix $c>c^*$ and $\Re_0>1$ for the convenience of our later content in this section.
\subsection{Construction of upper and lower solutions}
Define the following functions:
\begin{equation}
\label{UpLowSolution}\left\{
\begin{array}{l}
\vspace{2mm}
\displaystyle   S^+(\xi)=S_0,\\
\displaystyle   I^+(\xi) = \mathrm{e}^{\lambda_1 \xi}\wedge I_0,
\end{array}\right.
\ \
\left\{
\begin{array}{l}
\vspace{2mm}
\displaystyle   S^-(\xi)=S_0(1-M_1 \mathrm{e}^{\varepsilon_1 \xi})\vee 0,\\
\displaystyle   I^-(\xi)=\mathrm{e}^{\lambda_1\xi}(1-M_2\mathrm{e}^{\varepsilon_2 \xi})\vee 0,
\end{array}\right.
\end{equation}
where $I_0$ is a positive constant satisfying $I_0\geqslant \frac{\beta S_0 - \mu}{\alpha \mu}$, positive constants $M_i$ and $\varepsilon_i(i=1,2)$ will be determined in the following lemmas, the symbol $a\vee b : = \max\{a,b\}$ and $a\wedge b : = \min\{a,b\}$.
Next, we show that \eqref{UpLowSolution} are a pair upper and lower solutions of \eqref{WaveEqu}.

\begin{lemma}\label{SVUp}
The function $S^+(\xi)$ satisfies
\begin{equation}
\label{SVUpEqu}
c{S^+}'(\xi) \geq d_1\mathfrak{J}[S^+](\xi) + \Lambda -  \frac{\beta S^+(\xi)I^-(\xi)}{1+\alpha I^-(\xi)} - \mu_1 S^+(\xi).
\end{equation}
\end{lemma}
Since the proof of Lemma \ref{SVUp} above is obvious, we have omitted the details.
\begin{lemma}\label{IUp}
The function $ I^+(\xi)$ satisfies
\begin{equation}
\label{IUpEqu}
c {I^+}'(\xi) \geqslant d_2\mathfrak{J}[I^+](\xi) + \frac{\beta S^+(\xi) I^+(\xi)}{1+\alpha I^+(\xi)} - \mu_2 I^+(\xi).
\end{equation}
\end{lemma}

\begin{proof}
Note that
\[
d_2\mathfrak{J}[I^+](\xi) + \frac{\beta S^+(\xi) I^+(\xi)}{1+\alpha I^+(\xi)} - \mu_2 I^+(\xi)
\leqslant d_2\mathfrak{J}[I^+](\xi) + \beta S_0 I^+(\xi) - \mu_2 I^+(\xi),
\]
we only need to show that
\begin{equation}
\label{IUpEqu2}
c {I^+}'(\xi) \geqslant d_2\mathfrak{J}[I^+](\xi) + \beta S_0 I^+(\xi) - \mu_2  I^+(\xi).
\end{equation}
If $\xi < \frac{1}{\lambda_1}\ln I_0$, then $I^+(\xi) = \mathrm{e}^{\lambda_1 \xi}$,
and (\ref{IUpEqu2}) clearly holds by Lemma \ref{WaveSpeed}. If $\xi > \frac{1}{\lambda_1}\ln I_0$, then $I^+(\xi) = I_0 \geqslant \frac{\beta S_0 - \mu}{\alpha \mu}$, one has
\begin{align*}
d_2\mathfrak{J}[I^+](\xi) + \frac{\beta S^+(\xi) I^+(\xi)}{1+\alpha I^+(\xi)} - \mu_2 I^+(\xi) - c {I^+}'(\xi) \leqslant 0.
\end{align*}
This completes the proof.
\end{proof}

\begin{lemma}\label{LemLowS}
For each sufficiently small $0<\varepsilon_1<\lambda_1$ and $M_1>0$ is large enough, the function $S^-(\xi)$ satisfies
\begin{equation}\label{SLowEqu}
c{S^-}'(\xi) \leq d_1\mathfrak{J}[S^-](\xi) + \Lambda - \frac{\beta S^-(\xi)I^+(\xi)}{1+\alpha I^+(\xi)} - \mu_1 S^-(\xi)
\end{equation}
with $\xi\neq\frac{1}{\varepsilon_1}\ln\frac{1}{M_1}:=\mathfrak{B}_1$.
\end{lemma}
\begin{proof}
If $\xi>\mathfrak{X}_1$, then inequality (\ref{SLowEqu}) holds since $S^-(\xi) =0$, $S^-(\xi\pm\sin\theta) \geqslant 0$  and $S^-(\xi\pm\cos\theta) \geqslant 0$.
If $\xi<\mathfrak{B}_1$, then
\[
S^-(\xi)=S_0(1-M_1 \mathrm{e}^{\varepsilon_1 \xi}),\ \ S^-(\xi\pm\sin\theta)\geqslant S_0(1-M_1 \mathrm{e}^{\varepsilon_1 (\xi\pm\sin\theta)})
\]
and
\[
S^-(\xi\pm\cos\theta)\geqslant S_0(1-M_1 \mathrm{e}^{\varepsilon_1 (\xi\pm\cos\theta)}).
\]
Since $I^+(\xi)\geqslant 0$, we have
\begin{align*}
&\ d_1\mathfrak{J}[S^-](\xi) + \Lambda - \frac{\beta S^-(\xi)I^+(\xi)}{1+\alpha I^+(\xi)} - \mu_1 S^-(\xi) - c{S^-}'(\xi)\\
\geq &\ d_1\mathfrak{J}[S^-](\xi) + \Lambda - \beta S^-(\xi)I^+(\xi) - \mu_1 S^-(\xi) - c{S^-}'(\xi)\\
\geq &\ \mathrm{e}^{\varepsilon_1 \xi} S_0 \left[-M_1 (d_1\mathrm{e}^{\varepsilon_1\sin\theta} + d_1\mathrm{e}^{-\varepsilon_1\sin\theta} + d_1\mathrm{e}^{\varepsilon_1\cos\theta} + d_1\mathrm{e}^{-\varepsilon_1\cos\theta} - 4d_1 - \mu_1 - c \varepsilon_1) -\beta \mathrm{e}^{\lambda_1\xi}\mathrm{e}^{-\varepsilon_1\xi}\right].
\end{align*}
Select $0<\varepsilon_1<\lambda_1$ small enough such that
\[
d_1\mathrm{e}^{\varepsilon_1\sin\theta} + d_1\mathrm{e}^{-\varepsilon_1\sin\theta} + d_1\mathrm{e}^{\varepsilon_1\cos\theta} + d_1\mathrm{e}^{-\varepsilon_1\cos\theta} - 4d_1 - \mu_1 - c \varepsilon_1 < 0,
\]
and note that $\mathrm{e}^{(\lambda_1 - \varepsilon_1)\xi}\leq 1$ since $\xi<\mathfrak{B}_1<0$. Therefore, in order to ensure that inequality (\ref{SLowEqu}) holds,we need to select
\[
M_1 \geq -\frac{\beta}{d_1\mathrm{e}^{\varepsilon_1\sin\theta} + d_1\mathrm{e}^{-\varepsilon_1\sin\theta} + d_1\mathrm{e}^{\varepsilon_1\cos\theta} + d_1\mathrm{e}^{-\varepsilon_1\cos\theta} - 4d_1 - \mu_1 - c \varepsilon_1}
\]
large enough to ensure that the inequality  holds. Thus, this finishes the proof of Lemma \ref{LemLowS}.
\end{proof}

\begin{lemma}\label{LemLowI}
For each sufficiently small $0<\varepsilon_2<\lambda_1$  and $M_2>0$ is large enough, the function $ I^-(\xi)$ satisfies
\begin{equation}\label{lowIEqu}
c {I^-}'(\xi) \leqslant d_2\mathfrak{J}[I^-](\xi) + \frac{\beta S^-(\xi) I^-(\xi)}{1+\alpha I^-(\xi)} - \mu_2 I^-(\xi).
\end{equation}
with $\xi \neq \frac{1}{\varepsilon_2}\ln \frac{1}{M_2}:=\mathfrak{B}_2$.
\end{lemma}
\begin{proof}
If $\xi > \mathfrak{B}_2$, then inequality (\ref{lowIEqu}) holds since $I^-(\xi)=0$, $I^-(\xi\pm\sin\theta) \geqslant 0$  and $I^-(\xi\pm\cos\theta) \geqslant 0$. If $\xi < \mathfrak{B}_2$, then
\[
I^-(\xi)=\mathrm{e}^{\lambda_1\xi}(1-M_2\mathrm{e}^{\varepsilon_2 \xi}),\ \ I^-(\xi\pm\sin\theta)\geqslant \mathrm{e}^{\lambda_1(\xi\pm\sin\theta)}(1-M_2\mathrm{e}^{\varepsilon_2 (\xi\pm\sin\theta)})
\]
and
\[
I^-(\xi\pm\cos\theta)\geqslant \mathrm{e}^{\lambda_1(\xi\pm\cos\theta)}(1-M_2\mathrm{e}^{\varepsilon_2 (\xi\pm\cos\theta)}).
\]
Transforming the inequality (\ref{lowIEqu}) above is equivalent to the following inequality:
\begin{align}\label{I1}
& \ \beta S_0  I^-(\xi) - \frac{\beta S^-(\xi) I^-(\xi)}{1+\alpha I^-(\xi)}\\ \nonumber
\leq & \ d_2\mathfrak{J}[I^-] (\xi) - \mu_2  I^-(\xi) - c{ I^-}'(\xi) + \beta S_0  I^-(\xi).
\end{align}
For $(\ref{I1})$, we have
\begin{align}\label{Equ1}
\nonumber \beta S_0  I^-(\xi) - \frac{\beta S^-(\xi) I^-(\xi)}{1+\alpha I^-(\xi)} = \frac{\beta S_0  I^-(\xi)+\alpha\beta S_0  [I^-(\xi)]^{2}-\beta S^-(\xi) I^-(\xi)}{1+\alpha I^-(\xi)}.\\
\end{align}
Recalling that $\xi<\mathfrak{B}_2$, we can choose a sufficiently large $M_2$ so that
\[
 \  S^-(\xi)\rightarrow S_0.
\]
It follows from (\ref{Equ1}) that
\[
\beta S_0  I^-(\xi) - \frac{\beta S^-(\xi) I^-(\xi)}{1+\alpha I^-(\xi)} \leq \alpha\beta S_0 [ I^-(\xi)]^2.
\]
Furthermore, the right-hand of (\ref{I1}) is satisfy
\[
d_2\mathfrak{J}[ I^-](\xi) - \mu_2  I^-(\xi) - c{ I^-}'(\xi) + \beta S_0  I^-(\xi) \geq \mathrm{e}^{\lambda_1\xi}\Delta_c(\lambda_1,c) - M_2 \mathrm{e}^{\lambda_1+\varepsilon_2}\Delta_c(\lambda_1+\varepsilon_2,c).
\]
Employing the definition of $\Delta_c(\lambda)$ and the inequality $\alpha\beta S_0[ I^-(\xi)]^2\leq \alpha\beta S_0 \mathrm{e}^{2\lambda_1\xi}$, it can be observed through (\ref{WaveSpeedRel}) that $\Delta_c(\lambda_1+\varepsilon_2)<0$ holds for sufficiently small $\varepsilon_2>0$. To complete the proof, it suffices to demonstrate that
\[
\mathrm{e}^{(\lambda_1 - \varepsilon_2)\xi} \leq - M_2 \Delta_c(\lambda_1+\varepsilon_2).
\]
This inequality is guaranteed when $M_2$ is selected sufficiently large, as the left-hand side asymptotically approaches zero while the right-hand side diverges to positive infinity under the limit $M_2\rightarrow+\infty$. This ends the proof.
\end{proof}

\section{Existence of traveling wave solution}
\subsection{Truncated problem}
Let $X>\max\{-\mathfrak{B}_1,-\mathfrak{B}_2\}$. Define the following set
\begin{equation*}
\Gamma_X := \left\{(\phi, \psi)\in C([-X,X],\mathbb{R}^2)\left|
\begin{array}{l}
\vspace{2mm}
\displaystyle   S^-(\xi)\leq \phi(\xi) \leq S^+(\xi),\  \phi(-X)=S^-(-X),\\
\vspace{2mm}
\displaystyle   I^-(\xi)\leq \psi(\xi) \leq I^+(\xi),\ \ \psi(-X)=I^-(-X),\\
\displaystyle   \forall\xi\in[-X,X].\\
\end{array}\right.\right\}.
\end{equation*}
Obviously, $\Gamma_X$ is a nonempty bounded closed convex set in $C([-X,X],\mathbb{R}^2)$. For any $(\phi,\psi)\in C([-X,X],\mathbb{R}^2)$, it can be generalized as
\begin{equation*}
\hat{\phi}(\xi)=\left\{
\begin{array}{ll}
\displaystyle   \phi(X), &\mbox{for $\xi>X$,}
\\
\displaystyle   \phi(\xi), &\mbox{for $\xi\in[-X, X]$,}
\\
\displaystyle   S^-(\xi), &\mbox{for $\xi< -X$,}\\
\end{array}\right.\ \ \
\hat{\psi}(\xi)=\left\{
\begin{array}{ll}
\displaystyle   \psi(X), &\mbox{for $\xi>X$,}
\\
\displaystyle   \psi(\xi), &\mbox{for $\xi\in[-X, X]$,}
\\
\displaystyle   I^-(\xi), &\mbox{for $\xi< -X$.}\\
\end{array}\right.
\end{equation*}
Next, we consider the following truncation initial problem
\begin{equation}
\label{TruPro}\left\{
\begin{array}{l}
\vspace{2mm}
\displaystyle   cS'(\xi) + (4d_1+\mu_1+\alpha)S(\xi) = H_1(\phi,\psi),\\
\vspace{2mm}
\displaystyle   cI'(\xi) + (4d_2+\mu_2)I(\xi) = H_2(\phi,\psi),\\
\displaystyle   (S,I)(-X) = (S^-,I^-)(-X),
\end{array}\right.
\end{equation}
with
\begin{align*}
H_1(\phi,\psi) : = &\ d_{1}\phi(\xi + \sin\theta) + d_{1}\phi(\xi - \sin\theta) + d_{1}\phi(\xi + \cos\theta) + d_{1}\phi(\xi - \cos\theta)\\
 &\ + \Lambda - \frac{\beta \phi(\xi)\psi(\xi)}{1+\alpha \psi(\xi)} + \alpha\phi(\xi), 
\end{align*}
and
\begin{align*}
H_2(\phi,\psi) : =\ d_{2}\hat{\psi}(\xi + \cos\theta) + d_{2}\hat{\psi}(\xi - \cos\theta)+ \frac{\beta \phi(\xi)\psi(\xi)}{1+\alpha \psi(\xi)},
\end{align*}
where $(\phi,\psi)\in \Gamma_X$ and $\alpha$ is a sufficiently large constant that $H_1(\phi,\psi)$ does not decrease on $\phi(\xi)$. According to the theory of ordinary differential equations, the system of equation (\ref{TruPro}) has a unique solution $(S_X(\xi),I_X(\xi))\in C^1([-X,X],\mathbb{R}^2)$. Next, we define an operator
\[
\mathcal{P} = (\mathcal{P}_1,\mathcal{P}_2):\Gamma_X\rightarrow C^1\left([-X,X],\mathbb{R}^2\right)
\]
by
\[
S_X(\xi)=\mathcal{P}_1(\phi,\psi)(\xi)\ \ {\rm and}\ \ I_X(\xi)=\mathcal{P}_2(\phi,\psi)(\xi).
\]
Now, we choosing Schauder's fixed point theorem to prove that operator $\mathcal{P}=(\mathcal{P}_1,\mathcal{P}_2)$ admits a fixed point in $\Gamma_X$.

\begin{lemma}
The operator $\mathcal{P}=(\mathcal{P}_1,\mathcal{P}_2)$ maps $\Gamma_X$ into itself.
\end{lemma}
\begin{proof}
Firstly, we prove that for any $\xi\in[-X,X]$, $S^-(\xi)\leq S_X(\xi)$. If $\xi\in(\mathfrak{X}_1, X),$  then $S^-(\xi)=0$, which is the lower solution of the first equation in equation (\ref{TruPro}). If $\xi\in(-X,\mathfrak{B}_1),$ then $S^-(\xi)=S_0(1-M_1 \mathrm{e}^{\varepsilon_1 \xi})$, according to Lemma \ref{LemLowS}, we obtain
\begin{align*}
&c{S^-}'(\xi) + (4d_1+\mu_1+\alpha)S^-(\xi) - d_{1}(\phi(\xi + \sin\theta) + \phi(\xi - \sin\theta) + \phi(\xi + \cos\theta)
+\phi(\xi - \cos\theta))\\
&- \Lambda + \frac{\beta \phi(\xi)\psi(\xi)}{1+\alpha \psi(\xi)} - \alpha\phi(\xi)\\
\leq & c{S^-}'(\xi) - d_1\mathfrak{J}[S^-](\xi) - \Lambda + \mu_1 S^-(\xi) + \beta S^-(\xi) I^+(\xi)\\
\leq & 0.
\end{align*}
This means that $S^-(\xi)=S_0(1-M_1 \mathrm{e}^{\varepsilon_1 \xi})$ is the lower solution of the first equation in equation (\ref{TruPro}). Therefore, for any $\xi\in[-X,X],$ $S^-(\xi)\leq S_X(\xi)$.

Secondly, we prove that $S_X(\xi)\leq S^+(\xi) = S_0$ for any $\xi\in[-X,X].$ Actually, 
\begin{align*}
&c{S^+}'(\xi) + (4d_1+\mu_1+\alpha)S^+(\xi) - d_{1}\phi(\xi + \sin\theta) - d_{1}\phi(\xi - \sin\theta) - d_{1}\phi(\xi + \cos\theta) - d_{1}\phi(\xi - \cos\theta)\\
& - \Lambda + \frac{\beta \phi(\xi)\psi(\xi)}{1+\alpha \psi(\xi)} - \alpha\phi(\xi)\\
\geq & \beta S_0I^-(\xi)\\
\geq & 0,
\end{align*}
therefore, $S^+(\xi)=S_0$  is the upper solution of the first equation of the truncated initial problem (\ref{TruPro}), which gives us $S_X(\xi)\leq S_0$ for  $\forall \xi\in[-X,X].$

Similarly, we can use a similar approach to prove that for any $\xi\in[-X,X],$~$I^-(\xi)\leq I_X(\xi)\leq I^+(\xi)$. This completes the proof.
\end{proof}

\begin{lemma}

The operator $\mathcal{P}:\Gamma_X\rightarrow\Gamma_X$ is completely continuous.
\end{lemma}
\begin{proof}
Suppose $(\phi_i(\xi),\psi_i(\xi))\in\Gamma_X,\ i=1,2.$ Denote
\begin{align*}
S_{X,i}(\xi)=\mathcal{P}_1(\phi_i(\xi),\psi_i(\xi))\ \ {\rm and}\ \ I_{X,i}(\xi)=\mathcal{P}_2(\phi_i(\xi),\psi_i(\xi)).
\end{align*}
Next, prove that operator $\mathcal{P}$  is continuous. By direct calculation, we obtain
\[
S_X(\xi) = S^-(-X) \mathrm{e}^{-\frac{4d_1+\mu_1+\alpha}{c}(\xi+X)} + \frac{1}{c}\int_{-X}^\xi \mathrm{e} ^{\frac{4d_1+\mu_1+\alpha}{c}(\tau+X)}H_1(\phi,\psi)(\tau)\d \tau,
\]
and
\[
I_X(\xi) = I^-(-X) \mathrm{e}^{-\frac{4d_2+\mu_2}{c}(\xi+X)} + \frac{1}{c}\int_{-X}^\xi \mathrm{e}^{\frac{4d_2+\mu_2}{c}(\tau+X)}H_2(\phi,\psi)(\tau)\d \tau,
\]
where $H_i(\phi,\psi)(i=1,2)$ are defined in (\ref{TruPro}).
For any $(\phi_i, \psi_i)\in\Gamma_X$, $i=1,2$, we get
\begin{align*}
&\ |\phi_1(\xi)\psi_1(\xi) - \phi_2(\xi)\psi_2(\xi)|\\
\leq & \ |\phi_1(\xi)\psi_1(\xi) - \phi_1(\xi)\psi_2(\xi)|+|\phi_1(\xi)\psi_2(\xi) - \phi_2(\xi)\psi_2(\xi)|\\
\leq & \ S_0  \max_{\xi\in[-X,X]}|\psi_1(\xi)-\psi_2(\xi)| + I_{0}\max_{\xi\in[-X,X]}|\phi_1(\xi)-\phi_2(\xi)|.
\end{align*}
Since $S_X$ and $I_X$ are class of $C^1([-X,X])$, note that
\begin{align*}
&\ \left|c(S_{X,1}'(\xi)-S_{X,2}'(\xi))+(4d_1+\mu_1)(S_{X,1}(\xi)-S_{X,2}(\xi))\right|\\
\leq
&\ d_1|(\hat{\phi}_1(\xi+\sin\theta)-\hat{\phi}_2(\xi+\sin\theta))| + d_1|(\hat{\phi}_1(\xi-\cos\theta)-\hat{\phi}_2(\xi-\cos\theta))|\\
&\ +d_1|(\hat{\phi}_1(\xi+\cos\theta)-\hat{\phi}_2(\xi+\cos\theta))| +d_1|(\hat{\phi}_1(\xi-\sin\theta)-\hat{\phi}_2(\xi-\sin\theta))|\\
&\ + \beta|\phi_1(\xi)\psi_1(\xi) - \phi_2(\xi)\psi_2(\xi)|\\
\leq &\ \beta S_0  \max_{\xi\in[-X,X]}|\psi_1(\xi)-\psi_2(\xi)| + \left(4d_1+\beta I_{0 }\right)\max_{\xi\in[-X,X]}|\phi_1(\xi)-\phi_2(\xi)|.
\end{align*}
The discussion about  $I_X$ is similar. Hence, it is obvious that operator $\mathcal{P}$  is continuous. Next, we can get that $S_X'$ and $I_X'$  are bounded by equation (\ref{TruPro}).
Therefore,  operator $\mathcal{P}$ is compact and completely continuous. This is the end of the proof.
\end{proof}

By applying Schauder fixed point theorem, we obtain the following lemma.
\begin{lemma}
There exists $(S_X,I_X)\in\Gamma_X$ such that
\[
(S_X(\xi),I_X(\xi)) = \mathcal{P}(S_X,I_X)(\xi)
\]
for $\xi\in[-X,X]$.
\end{lemma}
Next, the following prior estimates for $(S_X,I_X)$ are presented. Define
\[
C^{1,1}([-X,X])=\{u\in C^{1}([-X,X]) \mid \bar{u},~\bar{u}'~\mathrm{are} ~\mathrm{Lipschitz}~\mathrm{continuous} \}                                                                                                                            \]
with the norm
\[\parallel\bar{u}\parallel_{C^{1,1}([-X,X])}=~\max_{x\in[-X,X]}|\bar{~u~}|+\max_{x\in[-X,X]}|\bar{~u~}'|+\sup_{x,y\in[-X,X],x\neq y}\frac{|\bar{~u~}'(x)-\bar{~u~}'(y)|}{|x-y|}
\]

\begin{lemma}\label{L-2.9}
There exists a constant $N>0$ such that
\[\|S_X\|_{C^{1,1}( [-T, T])}\leq N~ ~\mathrm{and}~ ~\|I_X\|_{C^{1,1}( [-T, T])}\leq N
\]
for~$0<Y<X$~and~$X>\max\{-\mathfrak{B}_1,-\mathfrak{B}_2\}$.
\end{lemma}
\begin{proof}
It is important to recall that $(S_X, I_X)$ is the fixed point of the operator $\mathcal{P}$,~Next
\begin{equation}\label{WaveEqu1}
\left\{\begin{array}{l}
\vspace{2mm}
\displaystyle c S'_{X}(\xi)=d_1\bar{S}_{X}(\xi+\sin\theta)+d_1\bar{S}_{X}(\xi+\cos\theta)+d_1\bar{S}_{X}(\xi-\sin\theta)+d_1\bar{S}_{X}(\xi-\cos\theta)-(4d_{1}+\mu_{1})S_{X}(\xi)\\
\hspace{1.6cm}+\Lambda - \displaystyle\frac{\beta S_{X}(\xi)I_{X}(\xi)}{1+\alpha I_{X}(\xi)} ,\\
\vspace{2mm}
c I'_{X}(\xi) =
d_2\bar{I}_{X}(\xi+\sin\theta)+d_2\bar{I}_{X}(\xi+\cos\theta)+d_2\bar{I}_{X}(\xi-\sin\theta)+d_2\bar{I}_{X}(\xi-\cos\theta)-(4d_{2}+\mu_{2})I_{X}(\xi)\\
\vspace{2mm}
\displaystyle
\hspace{1.6cm}+\displaystyle\frac{\beta S_{X}(\xi)I_{X}(\xi)}{1+\alpha I_{X}(\xi)} ,
\end{array}\right.
\end{equation}
where
\begin{equation*}
\bar{S}_{X}=\left\{
\begin{array}{ll}
\displaystyle   S_{X}(X), &\mbox{for $\xi>X$,}
\\
\displaystyle   S_{X}(\xi), &\mbox{for $ -X \leq\xi \leq X$,}
\\
\displaystyle   S^-(\xi), &\mbox{for $\xi< -X$,}\\
\end{array}\right.\ \ \
\bar{I}_{X}(\xi)=\left\{
\begin{array}{ll}
\displaystyle   {I}_{X}(X), &\mbox{for $\xi>X$,}
\\
\displaystyle  {I}_{X}(\xi), &\mbox{for $-X \leq\xi \leq X$,}
\\
\displaystyle   I^-(\xi), &\mbox{for $\xi< -X$.}\\
\end{array}\right.
\end{equation*}
It can be deduced from equation (\ref{WaveEqu}) ,~since $0 \leq S_X(\xi) \leq S_0$ and $0 \leq I_X(\xi) \leq I_{0}$ for all $\xi \in [-T, T]$, the following inequalities hold true
\[
|S_X'(\xi)| \leq N_{1}:=\frac{8d_1+\mu_1}{c}S_0+\frac{\Lambda}{c}+\frac{\beta S_0 I_{0}}{c},
\]
and
\[
|I_X'(\xi)| \leq N_{2}:= \frac{8d_2+\mu_2+\beta S_0 I_{0}}{c}.
\]
Following from the previous discussions ,  there exists a constant $N>0$ such that
\[
\|S_X\|_{C^1( [-T, T])} \leq N \text{ and } \|I_X\|_{C^1( [-T, T])} \leq N.
\]

Using the idea in \cite{ZhangWuIJB2019}, for fixed \(\theta\) ~and for any \(\xi, \rho \in  [-T, T]\), we have the following four cases

\textbf{Case I}. If $\xi+ \sin\theta>T,\ \rho+ \sin\theta>T,\ \xi+\cos\theta>T,\ \rho + \cos\theta> T$, then
\[
|\bar{S}_X(\xi + \sin\theta)+\bar{S}_X(\xi + \cos\theta)
-\bar{S}_X(\rho + \sin\theta)-\bar{S}_X(\rho + \cos\theta)|
=
|S_X(T)-S_X(T)| = 0.
\]

\textbf{Case II}. If $\xi+ \sin\theta<T,\ \rho+ \sin\theta<T,\ \xi+\cos\theta<T,\ \rho + \cos\theta<T$, then
\begin{align*}
&\ |\bar{S}_X(\xi + \sin\theta)+\bar{S}_X(\xi + \cos\theta)
-\bar{S}_X(\rho + \sin\theta)-\bar{S}_X(\rho + \cos\theta)|\\
=
&\ |{S}_X(\xi + \sin\theta)+{S}_X(\xi + \cos\theta)-{S}_X(\rho + \sin\theta)-{S}_X(\rho + \cos\theta)|.
\end{align*}

\textbf{Case III}. If $\xi + \sin\theta < T,\ \xi+\cos\theta < T,\ \rho +\sin\theta> T,\ \rho+\cos\theta > T$, then
\begin{align*}
&\ |\bar{S}_X(\xi + \sin\theta)+\bar{S}_X(\xi + \cos\theta)
-\bar{S}_X(\rho + \sin\theta)-\bar{S}_X(\rho + \cos\theta)|\\
=
&\ |{S}_X(\xi + \sin\theta)+{S}_X(\xi + \cos\theta)-2 S_X(T)|.
\end{align*}

\textbf{Case IV}. If $\xi + \sin\theta > T,\ \xi+\cos\theta > T,\ \rho +\sin\theta< T,\ \rho+\cos\theta < T$, then
\begin{align*}
&\ |\bar{S}_X(\xi + \sin\theta)+\bar{S}_X(\xi + \cos\theta)
-\bar{S}_X(\rho + \sin\theta)-\bar{S}_X(\rho + \cos\theta)|\\
=
&\ |2S_X(X)-{S}_X(\rho + \sin\theta)-{S}_X(\rho + \cos\theta)|.
\end{align*}
Overall, there exists a $N_1>0$ such that 
\[
|\bar{S}_X(\xi + \sin\theta)+\bar{S}_X(\xi + \cos\theta)
-\bar{S}_X(\rho + \sin\theta)-\bar{S}_X(\rho + \cos\theta)|\leq N_1|\xi - \rho|,
\]
for all \(\xi, \rho \in  [-T, T]\). Similarly, we have
\[
|\bar{I}_X(\xi + \sin\theta)+\bar{I}_X(\xi + \cos\theta)
-\bar{I}_X(\rho + \sin\theta)-\bar{I}_X(\rho + \cos\theta)|\leq N_2|\xi - \rho|,
\]
for all \(\xi, \rho \in  [-T, T]\). Furthermore,
\begin{align*}
 \Big|\displaystyle\frac{\beta S_{X}(\xi)I_{X}(\xi)}{1+\alpha I_{X}(\xi)} -\displaystyle\frac{\beta S_{X}(\rho)I_{X}(\rho)}{1+\alpha I_{X}(\rho)} \Big| 
&\leq {\frac{\beta}{\alpha}}|S_X(\xi)-S_X(\rho)|+\beta S_{0}|I_X(\xi)-I_X(\rho)|,
\end{align*}
for all \(\xi, \rho \in  [-T, T]\). Hence, there exists some constant \(N>0\) such that
\[
\|S_X\|_{C^{1,1}( [-T, T])} \leq N.
\]
Similarly,
\[
\|I_X\|_{C^{1,1}( [-T, T])} \leq N
\]
for any \(T < X\). This completes the proof.
\end{proof}

\subsection{Existence of Traveling Wave Solutions}
The ensuing discourse will commence with a presentation of the primary outcomes of this section.
\begin{theorem}
There exists a non-trivial travelling wave solution \((S(\xi), I(\xi))\) for any wave speed~\(c > c^*\) of the system (\ref{Model}) that satisfies
\[
S^-\leq S(\xi)\leq S^+ \text{ and } I^-\leq I(\xi)\leq I^+ \text{ in } \mathbb{R}.
\]
Also,
\[
\lim_{\xi\rightarrow -\infty}(S(\xi), I(\xi))=(S_0, 0) \text{ and } \lim_{\xi\rightarrow +\infty}(S(\xi), I(\xi))=(S^*, I^*).
\]
\end{theorem}
The proof of this theorem is divided into several steps.

\textbf{Step 1}
We show that system (\ref{Model}) admits a nontrivial traveling wave solution \((S(\xi), I(\xi))\) in \(\mathbb{R}\) and satisfying
\[
\lim_{\xi\rightarrow -\infty}(S(\xi), I(\xi))=(S_0, 0).
\]

It is demonstrated that the system (\ref{Model}) possesses a nontrivial travelling wave solution \((S(\xi), I(\xi))\) in \(\mathbb{R}\) and satisfies $\lim_{\xi\rightarrow -\infty}(S(\xi), I(\xi))=(S_0, 0).$

Suppose that \(\{X_n\}_{n = 1}^{+\infty}\) is an increasing sequence such that \(X_n > -\mathfrak{B}_2\) ~and \(X_n > Y\) , while \(X_n \to +\infty\)~ as \(n \to +\infty\)~for all \(n \in \mathbb{N}\), where~\(Y\) comes from Lemma \ref{L-2.9}. For the system of equations (\ref{TruPro}), denote \((S_n, I_n) \in \Gamma_{X_n}\) as a solution to this system of equations. Since the function \(I^{+}(\xi)\) is bounded in \([-X_N, X_N]\), for  \(\forall N \in \mathbb{N}\), the sequence
\[
\{S_n\}_{n \geq N} \text{ and } \{I_n\}_{n \geq N}
\]
are  uniformly bounded in \([-X_N, X_N]\). Next, by deduction of (\ref{TruPro}), this can be obtained
\[
\{S_n'\}_{n \geq N} \text{ and } \{I_n'\}_{n \geq N}
\]
are also uniformly bounded in \([-X_N, X_N]\). Using equation (\ref{TruPro}) again, we can denote \(S_n''(\xi)\) and \(I_n''(\xi)\) using \(S_n(\xi)\), \(I_n(\xi)\), \(S_n(\xi \pm \sin\theta)\), \(I_n(\xi \pm \cos\theta)\), \(S_n(\xi \pm 2\sin\theta)\) and \(I_n(\xi \pm 2\cos\theta)\), which give us
\[
\{S_n''\}_{n \geq N} \text{ and } \{I_n''\}_{n \geq N}
\]
are uniformly bounded in \([-X_N + 2\sin\theta, X_N - 2\sin\theta]\) and \([-X_N + 2\cos\theta, X_N - 2\cos\theta]\). By using the Arzela-Ascoli theorem (see Rudin 1991, Theorem A5), in order to extract the subsequence we can use a diagonal process, denoted \(\{S_{n_k}\}_{k \in \mathbb{N}}\) and \(\{I_{n_k}\}_{k \in \mathbb{N}}\) such that
\[
S_{n_k} \to S,\ I_{n_k} \to I,\ S_{n_k}' \to S' \text{ and } I_{n_k}' \to I' \text{ as } k \to +\infty
\]
For certain functions \(S\) and \(I\) in \(C^{1}(\mathbb{R})\), this holds consistently on any tight subinterval of \(\mathbb{R}\). Therefore, the solution of the system (\ref{WaveEqu})  \(\text{ in } \mathbb{R}\) is  \((S(\xi), I(\xi))\)
\[
S^{-}(\xi) \leq S(\xi) \leq S^{+}(\xi) \text{ and } I^{-}(\xi) \leq I(\xi) \leq I^{+}(\xi) 
\]
Moreover, according to the definition of equation (\ref{TruPro}), it follows:
\[
\lim_{\xi \to -\infty}(S(\xi), I(\xi)) = (S_0, 0).
\]

\textbf{Step 2}
 ~~In this step, we prove that the functions \(S(\xi)\)  and \(I(\xi)\) in \(\mathbb{R}\) satisfy \(0 < S(\xi) < S_0\) and \(I(\xi) > 0\).
We will first prove that \(I(\xi) > 0\) for  \(\forall ~\xi \in \mathbb{R}\). Suppose there exists a real number \(\xi_0\) such that \(I(\xi_0) = 0\) and \(\theta\neq k \pi\),  is required with \(\xi_<\xi_{0}\).
 Through the second equation of  (\ref{WaveEqu}), we get
\[
0 = d_1\mathfrak{J}[I](\xi_0) =I(\xi_0 + \sin\theta) + I(\xi_0 - \sin\theta)+ I(\xi_0 + \cos\theta)+I(\xi_0 - \cos\theta).
\]
As a result,  \(I(\xi_1 \pm \sin\theta) = I(\xi_1 \pm\cos\theta) = 0\) since \(I(\xi)\geq0\) in \(\mathbb{R}\). This contradicts our definition of \(\xi_0\).


Next, for  \(\theta\) above, suppose there exists a real number  \(S(\xi_{1}) = 0\) such that \(S(\xi_1) = 0\), then \(S'(\xi_1) = 0\) and \(\mathfrak{J}[S](\xi_1) \geq 0\). From the first equation of (\ref{WaveEqu}) , we have
\[
0 = d_1\mathfrak{J}[I](\xi_1) +\Lambda > 0 ,
\]
which is a contradiction. Consequently, \(S(\xi) > 0\) for all \(\xi \in \mathbb{R}\).

To show that \(S(\xi)<S_0\) for all \(\xi\in\mathbb{R}\), we assume that if there exists \(\xi_2\) such that \(S(\xi_2) = S_0\), it is easy to obtain
\[
0 = d_1\mathfrak{J}[S](\xi_2)-\frac{\beta S(\xi_2)I(\xi_2)}{1+\alpha I(\xi_2)}<0.
\]
This contradiction leads to \(S(\xi)<S_0\) for all \(\xi\in\mathbb{R}\).

\textbf{Step 3 }Boundedness of traveling wave solutions \(S(\xi)\) and \(I(\xi)\) in \(\mathbb{R}\).

We are able to easily obtain \(0<S(\xi)<S_{0}\) and \(0<I(\xi)<I_{0}\), where $I_0= \displaystyle\frac{\beta S_0 - \mu_{2}}{\alpha \mu_{2}}$. Therefore, the proof of the above step is simple. So we just give the conclusion ignoring the details of the proof.

In order for the Lyapunov function we construct to be meaningful we need the following lemma.
\begin{lemma}\label{10}
The functions \(\displaystyle\frac{I(\xi\pm\sin\theta)}{I(\xi)}\), \(\displaystyle\frac{I(\xi\pm\cos\theta)}{I(\xi)}\) and \(\displaystyle\frac{I'(\xi)}{I(\xi)}\) are bounded in \(\mathbb{R}\).
\end{lemma}
\begin{proof}
Due to the uncertainty of the positive and negative values of  \(\sin\theta\) and \(\cos\theta\), we will categorize them into four situations. We assume  \(\theta\neq \displaystyle\frac{\pi}{2}+2k\pi~~(k\in  Z)\), firstly \(\sin\theta>0~,~\cos\theta>0\), secondly \(\sin\theta>0~,~\cos\theta<0\), thirdly \(\sin\theta<0~,~\cos\theta<0\), and finally \(\sin\theta<0~,~\cos\theta>0\). The proof ideas for the above four situations are the same. We will use the first situation as an example to demonstrate, and leave the remaining three situations for readers to prove themselves.

 From the second equation in the traveling wave system (\ref{WaveEqu}), we get
\[
cI'(\xi)+(4d_2+\mu_2)I(\xi)=d_2I(\xi + \sin\theta)+d_2I(\xi -\sin\theta)+d_2I(\xi +\cos\theta)+d_2I(\xi -\cos\theta)+
\frac{\beta S(\xi)I(\xi)}{1+\alpha I(\xi)}>0.
\]
Denote \(W(\xi)=e^{\nu\xi}I(\xi)\), where \(\nu=(4d_2+\mu_2)/c\). As a result
\[
cW'(\xi)=e^{\nu\xi}(cI'(\xi)+(4d_2+\mu_2)I(\xi))>0,
\]
thus \(W(\xi)\) is increasing on \(\xi\). Then, \(W(\xi - \sin\theta)<W(\xi)\), this means,
\[
\frac{I(\xi - \sin\theta)}{I(\xi)}<e^{\nu} ~~~ \forall ~ \xi\in\mathbb{R}.
\]
Note that
\begin{equation}\label{3.1}
\begin{aligned}
\left[e^{\nu\xi}I(\xi)\right]' &= \frac{1}{c}e^{\nu\xi}\left[d_2I(\xi + \sin\theta)+d_2I(\xi -\sin\theta)+d_2I(\xi +\cos\theta)+d_2I(\xi -\cos\theta)+
\frac{\beta S(\xi)I(\xi)}{1+\alpha I(\xi)}\right]\\
&>\frac{d_2}{c}e^{\nu\xi}I(\xi + \sin\theta).
\end{aligned}
\end{equation}
Integrating \ref{3.1} on \([\xi, \xi + \sin\theta]\) and utilizing the fact that \(e^{\nu\xi}I(\xi)\) increases, we obtain
\begin{align*}
e^{\nu(\xi + \sin\theta)}I(\xi +  \sin\theta)&>e^{\nu\xi}I(\xi)+\frac{d_2}{c}\int_{\xi}^{\xi +  \sin\theta}e^{\nu s}I(s +  \sin\theta)ds\\
&>e^{\nu\xi}I(\xi)+\frac{d_2}{c}\int_{\xi}^{\xi +  \sin\theta}e^{\nu(\xi +  \sin\theta)}I(\xi +  \sin\theta)e^{-\nu\sin\theta}ds\\
&=e^{\nu\xi}\left[I(\xi)+\frac{d_2}{c}I(\xi +  \sin\theta) \sin\theta\right].
\end{align*}
By (\ref{3.1}), we have
\begin{equation}\label{3.2}
\begin{aligned}
\left[e^{\nu\xi}I(\xi)\right]'&>\left(\frac{d_2}{c}\right)^2e^{-2\nu\sin\theta}e^{\nu(\xi + \sin\theta)}I(\xi + \sin\theta)\sin\theta.
\end{aligned}
\end{equation}
Integrating (\ref{3.2}) over \(\left[\xi - \displaystyle \frac{\sin\theta}{2}, \xi\right]\) yields
\begin{align*}
e^{\nu\xi}I(\xi)&>\left(\displaystyle \frac{d_2}{c}\right)^2e^{-2\nu\sin\theta}\sin\theta\int_{\xi - \frac{\sin\theta}{2}}^{\xi}e^{\nu(s + \sin\theta)}I(s + \sin\theta)ds\\
 &>\left(\frac{d_2}{c}\right)^2\frac{e^{-2\nu\sin\theta}\sin^{2}\theta }{2}e^{\nu(\xi + \frac{\sin\theta}{2})}I\left(\xi + \frac{\sin\theta}{2}\right),
\end{align*}
that is
\[
\frac{I\left(\xi + \frac{\sin\theta}{2}\right)}{I(\xi)}<2\left(\frac{c}{d_2}\right)^2e^{\frac{3}{2}\nu\sin\theta}\sin^{2}\theta ~~\text{ for all } \xi \in \mathbb{R}.
\]
Similarly, integrating (\ref{3.2}) over \(\left[\xi, \xi + \frac{\sin\theta}{2}\right]\), we have
\[
\frac{I(\xi + 1)}{I\left(\xi + \frac{\sin\theta}{2}\right)}<2\left(\frac{c}{d_2}\right)^2e^{\frac{3}{2}\nu\sin\theta}\sin^{2}\theta~~ \text{ for all } \xi \in \mathbb{R}.
\]
Thus,
\[
\frac{I(\xi + \sin\theta)}{I(\xi)}=\frac{I\left(\xi + \frac{\sin\theta}{2}\right)}{I(\xi)}\frac{I(\xi + \sin\theta)}{I\left(\xi + \frac{\sin\theta}{2}\right)}<4\left(\frac{c}{d_2}\right)^4e^{3\nu\sin\theta}\sin^{4}\theta~~ \text{ for all } \xi \in \mathbb{R}.
\]
Similarly, the proof for $\displaystyle\frac{I(\xi \pm \cos\theta)}{I(\xi)}$ is similar.

By the second equation of (\ref{WaveEqu}), it follows that
\begin{align*}
cI'(\xi)&=d_{2}I(\xi + \sin\theta)+d_{2}I(\xi - \sin\theta)+d_{2}I(\xi + \cos\theta)+d_{2}I(\xi - \cos\theta)+\frac{\beta S(\xi){I(\xi)}}{1+\alpha I(\xi)}-(4d_2+\mu_2)I(\xi)\\
&\leq\d_{2}\frac{I(\xi + \sin\theta)}{I(\xi)}+d_{2}\frac{I(\xi - \sin\theta)}{I(\xi)}+d_{2}\frac{I(\xi + \cos\theta)}{I(\xi)}+d_{2}\frac{I(\xi - \cos\theta)}{I(\xi)}+\beta S(\xi)-(4d_2+\mu_2)\\
&\leq\d_{2}\frac{I(\xi + \sin\theta)}{I(\xi)}+d_{2}\frac{I(\xi - \sin\theta)}{I(\xi)}+d_{2}\frac{I(\xi + \cos\theta)}{I(\xi)}+d_{2}\frac{I(\xi - \cos\theta)}{I(\xi)}+\beta S_{0}-(4d_2+\mu_2),
\end{align*}
which gives us \(\displaystyle\frac{I'(\xi)}{I(\xi)}\) is bounded in \(\mathbb{R}\). This proof is over.
\end{proof}

Step 4
\begin{proof}
Define the following Lyapunov functional
\[
L(\xi) = W(\xi) + d_1 S^*(V_1(\xi)+V_2(\xi)) + d_2 I^* (U_1(\xi)+U_2(\xi)),
\]
where
\[
\left\{
\begin{array}{l}
\vspace{1mm}
\displaystyle   W(\xi) = cS^* g\left(\frac{S(\xi)}{S^*}\right) + cI^* g\left(\frac{I(\xi)}{I^*}\right),\\
\vspace{1mm}
\displaystyle   V_1(\xi) = \int_0^{\sin\theta} g\left(\frac{S(\xi-\varsigma)}{S^*}\right)\d \varsigma - \int_{-\sin\theta}^0 g\left(\frac{S(\xi-\varsigma)}{S^*}\right)\d \varsigma,\\
\vspace{1mm}
\displaystyle   V_2(\xi) = \int_0^{\cos\theta} g\left(\frac{S(\xi-\varsigma)}{S^*}\right)\d \varsigma - \int_{-\cos\theta}^0 g\left(\frac{S(\xi-\varsigma)}{S^*}\right)\d \varsigma,\\
\vspace{1mm}
\displaystyle   U_1(\xi) = \int_0^{\sin\theta} g\left(\frac{I(\xi-\varsigma)}{I^*}\right)\d \varsigma - \int_{-\sin\theta}^0 g\left(\frac{I(\xi-\varsigma)}{I^*}\right)\d \varsigma,\\
\vspace{1mm}
\displaystyle   U_2(\xi) = \int_0^{\cos\theta} g\left(\frac{I(\xi-\varsigma)}{I^*}\right)\d \varsigma - \int_{-\cos\theta}^0 g\left(\frac{I(\xi-\varsigma)}{I^*}\right)\d \varsigma.
\end{array}\right.
\]

\begin{align*}
\frac{\d W(\xi)}{\d \xi} =& c \left(1-\frac{S^*}{S(\xi)}\right)\frac{\d S(\xi)}{\d \xi} + c \left(1-\frac{I^*}{I(\xi)}\right)\frac{\d I(\xi)}{\d \xi}\\
=& \left(1-\frac{S^*}{S(\xi)}\right)\left(d_1 \mathfrak{D}[S](\xi) + \Lambda - \frac{\beta S(\xi)I(\xi)}{1+\alpha I(\xi)} - \mu_1 S(\xi)\right)\\
& + \left(1-\frac{I^*}{I(\xi)}\right)\left(d_2 \mathfrak{D}[I](\xi) + \frac{\beta S(\xi)I(\xi)}{1+\alpha I(\xi)} - \mu_2 I(\xi)\right)\\
=& d_1 \left(1-\frac{S^*}{S(\xi)}\right) \mathfrak{D}[S](\xi) + d_2 \left(1-\frac{I^*}{I(\xi)}\right) \mathfrak{D}[I](\xi) + \mathfrak{A}(\xi),
\end{align*}
where
\[
\mathfrak{A}(\xi) = \left(1-\frac{S^*}{S(\xi)}\right)\left(\Lambda - \frac{\beta S(\xi)I(\xi)}{1+\alpha I(\xi)} - \mu_1 S(\xi)\right) + \left(1-\frac{I^*}{I(\xi)}\right)\left(\frac{\beta S(\xi)I(\xi)}{1+\alpha I(\xi)} - \mu_2 I(\xi)\right).
\]
Using the fact that
\[
\Lambda = \frac{\beta S^*I^*}{1+\alpha I^*} - \mu_1 S^*\ \ \ \textrm{and}\ \ \ \frac{\beta S^*I^*}{1+\alpha I^*} = \mu_2 I^*,
\]
we obtain that
\begin{align*}
\mathfrak{A}(\xi) = \frac{\beta S^*I^*}{1+\alpha I^*}\left(3 - \frac{S^*}{S(\xi)} - \frac{S(\xi)(1+\alpha I^*)}{S^* (1+\alpha I(\xi))} - \frac{1+\alpha I(\xi)}{1+\alpha I^*}\right) - \frac{\alpha \beta S^* (I(\xi)-I^*)^2}{(1+\alpha I(\xi))(1+\alpha I^*)^2}.
\end{align*}
Now, we deal other parts of the Lyapunov functional, firstly, one has that
\begin{align*}
\frac{\d V_1(\xi)}{\d \xi} = &\ \frac{\d}{\d \xi} \left[\int_0^{\sin\theta} g\left(\frac{S(\xi-\varsigma)}{S^*}\right)\d \varsigma - \int_{-\sin\theta}^0 g\left(\frac{S(\xi-\varsigma)}{S^*}\right)\d \varsigma\right]\\
= &\  \int_0^{\sin\theta} \frac{\d}{\d \xi}g\left(\frac{S(\xi-\varsigma)}{S^*}\right)\d \varsigma - \int_{-\sin\theta}^0 \frac{\d}{\d \xi}g\left(\frac{S(\xi-\varsigma)}{S^*}\right)\d \varsigma\\
= &\  -\int_0^{\sin\theta} \frac{\d}{\d \varsigma}g\left(\frac{S(\xi-\varsigma)}{S^*}\right)\d \varsigma + \int_{-\sin\theta}^0 \frac{\d}{\d \varsigma}g\left(\frac{S(\xi-\varsigma)}{S^*}\right)\d \varsigma\\
= &\ 2 g\left(\frac{S(\xi)}{S^*}\right) - g\left(\frac{S(\xi+\sin\theta)}{S^*}\right) - g\left(\frac{S(\xi-\sin\theta)}{S^*}\right).
\end{align*}
Similarly, we have
\[
\frac{\d V_2(\xi)}{\d \xi} = 2 g\left(\frac{S(\xi)}{S^*}\right) - g\left(\frac{S(\xi+\cos\theta)}{S^*}\right) - g\left(\frac{S(\xi-\cos\theta)}{S^*}\right),
\]
\[
\frac{\d U_1(\xi)}{\d \xi} = 2 g\left(\frac{I(\xi)}{I^*}\right) - g\left(\frac{I(\xi+\sin\theta)}{I^*}\right) - g\left(\frac{I(\xi-\sin\theta)}{I^*}\right),
\]
and
\[
\frac{\d U_2(\xi)}{\d \xi} = 2 g\left(\frac{I(\xi)}{I^*}\right) - g\left(\frac{I(\xi+\cos\theta)}{I^*}\right) - g\left(\frac{I(\xi-\cos\theta)}{I^*}\right).
\]
With the help of zero trick $\ln \frac{a}{b} + \ln \frac{b}{a} = 0$ for positive $a$ and $b$, we obtained the following equations through calculation,
\begin{align*}
&d_1 \left(1-\frac{S^*}{S(\xi)}\right) \mathfrak{D}[S](\xi) + d_1 S^* \frac{\d V_1(\xi)}{\d \xi}\\
= & -g\left(\frac{S(\xi-\sin\theta)}{S(\xi)}\right) -g\left(\frac{S(\xi+\sin\theta)}{S(\xi)}\right) -g\left(\frac{S(\xi-\cos\theta)}{S(\xi)}\right) -g\left(\frac{S(\xi+\cos\theta)}{S(\xi)}\right),
\end{align*}
and
\begin{align*}
&d_2 \left(1-\frac{I^*}{I(\xi)}\right) \mathfrak{D}[I](\xi) + d_2 I^* \frac{\d V_2(\xi)}{\d \xi}\\
= & -g\left(\frac{I(\xi-\sin\theta)}{I(\xi)}\right) -g\left(\frac{I(\xi+\sin\theta)}{I(\xi)}\right) -g\left(\frac{I(\xi-\cos\theta)}{I(\xi)}\right) -g\left(\frac{I(\xi+\cos\theta)}{I(\xi)}\right).
\end{align*}
Overall, we compute the derivation of Lyapunov function $L(\xi)$ along system (\ref{WaveEqu}) as follows
\begin{align*}
\frac{\d L(\xi)}{\d \xi} = &\frac{\beta S^*I^*}{1+\alpha I^*}\left(3 - \frac{S^*}{S(\xi)} - \frac{S(\xi)(1+\alpha I^*)}{S^* (1+\alpha I(\xi))} - \frac{1+\alpha I(\xi)}{1+\alpha I^*}\right) - \frac{\alpha \beta S^* (I(\xi)-I^*)^2}{(1+\alpha I(\xi))(1+\alpha I^*)^2}\\
&-g\left(\frac{S(\xi-\sin\theta)}{S(\xi)}\right) -g\left(\frac{S(\xi+\sin\theta)}{S(\xi)}\right) -g\left(\frac{S(\xi-\cos\theta)}{S(\xi)}\right) -g\left(\frac{S(\xi+\cos\theta)}{S(\xi)}\right)\\
&-g\left(\frac{I(\xi-\sin\theta)}{I(\xi)}\right) -g\left(\frac{I(\xi+\sin\theta)}{I(\xi)}\right) -g\left(\frac{I(\xi-\cos\theta)}{I(\xi)}\right) -g\left(\frac{I(\xi+\cos\theta)}{I(\xi)}\right).
\end{align*}
\end{proof}

\section{Nonexistence of Traveling Wave Solutions}
In this section, we devote ourselves to the nonexistence of traveling wave solutions. First, we show that there exists a nontrivial positive solution \((S(\xi), I(\xi))\) satisfying the asymptotic boundary conditions  (\ref{bj1}) and  (\ref{bj2})  for the system of equations (\ref{WaveEqu}), when \(c > 0\).
\begin{lemma}\label{lemma5.1}
Assume that \(\Re_0>1\) and there exists a nontrivial solution \((S(\xi), I(\xi))\) of system (\ref{Model}) satisfying the asymptotic boundary conditions (\ref{bj1}) and  (\ref{bj2}). Then, \(c > 0\), where \(c\) is defined in (\ref{dy2.1}).
\end{lemma}
\begin{proof}
Next, we will use the condition \(\Re_0 > 1\). Suppose that \(c\leq0\). Due to \(S(\xi)\to S_0\) and \(I(\xi)\to0\) as \(\xi\to-\infty\), then exists a \(\xi^* < 0\) such that
\begin{align}\label{5.1}
cI'(\xi)&\geq d_2[I(\xi + \sin\theta)+I(\xi - \sin\theta)+I(\xi + \cos\theta)+I(\xi - \cos\theta)-4I(\xi)]\nonumber\\
&+\frac{\beta S_0+\mu_2}{2}I(\xi)-\mu_2I(\xi).
\end{align}
For \(\xi < \xi^*\), using inequality (\ref{5.1}) we get
\begin{align}\label{5.2}
cI'(\xi)&\geq d_2[I(\xi + \sin\theta)+I(\xi - \sin\theta)+I(\xi + \cos\theta)+I(\xi - \cos\theta)-4I(\xi)]+\frac{\beta S_0-\mu_2}{2}I(\xi).
\end{align}
Denote \(\displaystyle\vartheta=\frac{\beta S_0-\mu_2}{2}\) and \(E(\xi)=\displaystyle\int_{-\infty}^{\xi}I(y)dy\) for \(\xi\in\mathbb{R}\), note that \(\vartheta > 0\) since \(\Re_0 > 1\). Integrating inequality (\ref{5.1}) from \(-\infty\) to \(\xi\) and using \(I(-\infty)=0\), we  get
\begin{align}
cI(\xi)&\geq d_2[E(\xi + \sin\theta)+E(\xi - \sin\theta)+E(\xi + \cos\theta)+E(\xi - \cos\theta)-4E(\xi)]+\vartheta E(\xi)  ~~~\text{ for } \xi < \xi^*.
\end{align}
Next, integrating inequality (\ref{5.1}) from \(-\infty\) to \(\xi\), then for any $\xi < \xi^*$, one has that
\begin{align}\label{5.4}
cE(\xi)&\geq d_2\left(\int_{\xi}^{\xi + \sin\theta}E(\tau)d\tau-\int_{\xi - \sin\theta}^{\xi}E(\tau)d\tau+\int_{\xi}^{\xi + \cos\theta}E(\tau)d\tau-\int_{\xi - \cos\theta}^{\xi}E(\tau)d\tau\right)+\vartheta\int_{-\infty}^{\xi}E(\tau)d\tau.
\end{align}
Since \(E(\xi)\) is strictly increasing in \(\mathbb{R}\) and \(c\leq0\), we can deduce the conclusion:
\[
0\geq cE(\xi)\geq d_2\left(\int_{\xi}^{\xi + \sin\theta}E(\tau)d\tau-\int_{\xi - \sin\theta}^{\xi}E(\tau)d\tau+\int_{\xi}^{\xi + \cos\theta}E(\tau)d\tau-\int_{\xi - \cos\theta}^{\xi}E(\tau)d\tau\right)+\vartheta\int_{-\infty}^{\xi}E(\tau)d\tau >0.
\]
\end{proof}
The above equation is a contradiction . Hence, \(c > 0\). This proof is complete. \(\square\)

Next, we will use two - sided Laplace method to prove the nonexistence of traveling wave solutions.
\begin{theorem}
Assume that \(\Re_0 > 1\) and \(c < c^*\). Then, there is no nontrivial solution \((S(\xi), I(\xi))\) of system (\ref{Model}) satisfying the asymptotic boundary conditions (\ref{bj1}) and (\ref{bj2}).
\end{theorem}
\begin{proof}
By finding contradictory methods, assume that system (\ref{Model}) has a nontrivial positive solution \((S(\xi), I(\xi))\) that satisfies asymptotic boundary conditions (\ref{bj1}) and (\ref{bj2}). Next, \(c > 0\) By Lemma \ref{lemma5.1} and
\[
S(\xi)\to S_0 \text{ and } I(\xi)\to0 \text{ as } \xi\to-\infty.
\]
Let \(\vartheta = \displaystyle\frac{\beta S_0-\mu_2}{2}\) and \(E(\xi)=\int_{-\infty}^{\xi}I(y)dy\) for \(\xi\in\mathbb{R}\). From the proof of lemma \ref{lemma5.1}, it can be concluded that there exists a \(\xi^* < 0\), such that
\[
\begin{aligned}
cE(\xi)&\geq d_2\left(\int_{\xi}^{\xi + \sin\theta}E(\tau)d\tau-\int_{\xi - \sin\theta}^{\xi}E(\tau)d\tau+\int_{\xi}^{\xi + \cos\theta}E(\tau)d\tau-\int_{\xi - \cos\theta}^{\xi}E(\tau)d\tau\right)+\vartheta\int_{-\infty}^{\xi}E(\tau)d\tau ~~\text{ for } \xi < \xi^*.
\end{aligned}
\]
Due to the strict increment of \(E(\xi)\) in \(\mathbb{R}\) , we can get
\[
d_2\left(\int_{\xi}^{\xi + \sin\theta}E(\tau)d\tau-\int_{\xi - \sin\theta}^{\xi}E(\tau)d\tau+\int_{\xi}^{\xi + \cos\theta}E(\tau)d\tau-\int_{\xi - \cos\theta}^{\xi}E(\tau)d\tau\right)>0.
\]
Thus,
\begin{align}\label{5.5}
cE(\xi)&\geq\vartheta\int_{-\infty}^{\xi}E(\tau)d\tau  \text{ for } \xi < \xi^*.
\end{align}
Therefore, there exists some constant \(\eta > 0\) such that
\begin{align}
\vartheta\eta E(\xi - \eta)&\leq cE(\eta) \text{ for } \xi < \xi^*.
\label{eq:4.6}
\end{align}
Furthermore, there exists a sufficiently large value of \(h > 0\) and \(\epsilon_0\in(0, 1)\) such that
\begin{align}
E(\xi - h)&\leq\epsilon_0E(\xi) \text{ for } \xi < \xi^*.
\end{align}
Set
\[
\mu_0:=\frac{1}{h}\ln\frac{1}{\epsilon_0} \text{ and } F(\xi):=E(\xi)e^{-\mu_0\xi}.
\]
We have
\[
F(\xi - h)=E(\xi - h)e^{-\mu_0(\xi - h)}<\epsilon_0e(\xi)e^{-\mu_0(\xi - h)} = F(\xi) \text{ for } \xi < \xi^*,
\]
This means that \(F(\xi)\) is bounded when \(\xi\to-\infty\).
Since \(\int_{-\infty}^{\infty}I(\xi)d\xi < \infty\), we obtain that
\[
\lim_{\xi\to\infty}F(\xi)=\lim_{\xi\to\infty}E(\xi)e^{-\mu_0\xi}=0.
\]

According to  the second equation of (\ref{WaveEqu}), we get
\[
cI'(\xi)\leq d_2[I(\xi + \sin\theta)+I(\xi - \sin\theta)+I(\xi + \cos\theta)+I(\xi - \cos\theta)-4I(\xi)]+\beta S_0I(\xi)-\mu_2I(\xi),
\]
integrating over \((-\infty, \xi)\) gives us
\[
cI(\xi)\leq d_2[E(\xi + \sin\theta)+E(\xi - \sin\theta)+E(\xi + \cos\theta)+E(\xi - \cos\theta)-4E(\xi)]+\beta S_0E(\xi)-\mu_2E(\xi).
\]
Thus, we can deduce that
\[
\sup_{\xi\in\mathbb{R}}\{I(\xi)e^{-\mu_0\xi}\}<+\infty \text{ and } \sup_{\xi\in\mathbb{R}}\{I'(\xi)e^{-\mu_0\xi}\}<+\infty.
\]
Define the following two-sided Laplace transform for \(I(\cdot)\), where \(\lambda\in\mathbb{C}\) and \(0 < \text{Re}\lambda < \mu_0\)
\[
\mathcal{L}(\lambda):=\int_{-\infty}^{\infty}e^{-\lambda\xi}I(\xi)d\xi.
\]
Note that
\begin{align*}
&\int_{-\infty}^{\infty}e^{-\lambda\xi}[I(\xi + \sin\theta)+I(\xi - \sin\theta)+I(\xi + \cos\theta)+I(\xi - \cos\theta)]d\xi\\
&=e^{\sin\theta}\int_{-\infty}^{\infty}e^{-\lambda(\xi + \sin\theta)}I(\xi + \sin\theta)d\xi
+e^{-\sin\theta}\int_{-\infty}^{\infty}e^{-\lambda(\xi - \sin\theta)}I(\xi - \sin\theta)d\xi\\
&+e^{\cos\theta}\int_{-\infty}^{\infty}e^{-\lambda(\xi + \cos\theta)}I(\xi + \cos\theta)d\xi
+e^{-\cos\theta}\int_{-\infty}^{\infty}e^{-\lambda(\xi - \cos\theta)}I(\xi - \cos\theta)d\xi\\
&=(e^{\sin\theta}+e^{-\sin\theta})\mathcal{L}(\lambda)+(e^{\cos\theta}+e^{-\cos\theta})\mathcal{L}(\lambda)
\end{align*}
and
\[
\int_{-\infty}^{\infty}e^{-\lambda\xi}I'(\xi)d\xi=\left.e^{\lambda}I(\xi)\right|_{-\infty}^{\infty}-\int_{-\infty}^{\infty}I(\xi)de^{-\lambda\xi}=\lambda\mathcal{L}(\lambda).
\]
From the second equation of (\ref{WaveEqu}), we can obtain
\begin{align}
d_2\mathfrak{J}[I](\xi)+\beta S_0I(\xi)-\mu_2I(\xi)-cI'(\xi)&=\beta S_0I(\xi)-\frac{\beta S(\xi)I(\xi)}{1 + \alpha I(\xi)}.
\label{eq:5.8}
\end{align}
By performing two-sided Laplace transform on formula (\ref{eq:5.8}), we can obtain
\begin{align}\label{eq:5.9}
\Delta(\lambda, c)\mathcal{L}(\lambda)&=\int_{-\infty}^{\infty}e^{-\lambda\xi}[\beta S_0I(\xi)-\frac{\beta S(\xi)I(\xi)}{1 + \alpha I(\xi)}]d\xi.
\end{align}
From the proof in Lemma \ref{lemma5.1}, it can be concluded that when \(\xi\to-\infty\), we have
\begin{align*}
[\beta S_0I(\xi)-\frac{\beta S(\xi)I(\xi)}{1 + \alpha I(\xi)}]e^{-2\mu_0\xi}\leq \alpha\beta S_0 [ I^-(\xi)]^2e^{-2\mu_0\xi}
&\leq\alpha\beta S_0\left(\sup_{\xi\in\mathbb{R}}\{I(\xi)e^{-\mu_0\xi}\}\right)^2\\
&\leq+\infty.
\end{align*}
Therefore, we can conclude that
\begin{align}\label{eq:5.10}
\sup_{\xi\in\mathbb{R}}[\beta S_0I(\xi)-\frac{\beta S(\xi)I(\xi)}{1 + \alpha I(\xi)}]e^{-2\mu_0\xi}&<+\infty.
\end{align}

By the property of Laplace transform (Widder 1941), either there exists a real number \(\mu_0\) such that \(\mathcal{L}(\lambda)\) is analytic for \(\lambda\in\mathbb{C}\) with \(0 < \text{Re}\lambda < \mu_0\) and \(\lambda = \mu_0\) is singular point of \(\mathcal{L}(\lambda)\), or \(\mathcal{L}(\lambda)\) is well defined for \(\lambda\in\mathbb{C}\) with \(\text{Re}\lambda > 0\). Furthermore, the two Laplace integrals can be analytically continued to the whole right half line; otherwise, the integral on the left of (\ref{eq:5.9}) has singularity at \(\lambda = \mu_0\) and it is analytic for all \(\lambda < \mu_0\). However, it follows from (\ref{eq:5.10}) that the integral on the right of (\ref{eq:5.9}) is actually analytic for all \(\lambda\leq2\mu_0\), a contradiction. Thus, (\ref{eq:5.9}) holds for all \(\text{Re}\lambda > 0\). From Lemma 1, \(\Delta(\lambda, c)>0\) for all \(\lambda > 0\) and by the definition of \(\Delta(\lambda, c)\) in (\ref{CharEqu}), we know that \(\Delta(\lambda, c)\to\infty\) as \(\lambda\to\infty\), which is a contradiction with Eq (\ref{eq:5.9}). This ends the proof. \(\square\)
\end{proof}

\section*{Acknowledgments}
This study was supported by the National Natural Science Foundation of China (nos. 12371490, 12401638), the
Natural Science Foundation of Heilongjiang Province, PR China (nos. YQ2024A011, JQ2023A005), the Outstanding
Youth Fund of Heilongjiang University, PR China (no. JCL202203)..
\bibliographystyle{elsarticle-num} 
\bibliography{MyBib}



\end{document}